\documentclass[11pt,a4paper]{article}
\usepackage{latexsym,  bm}

\usepackage{amsfonts}
\usepackage{amsthm}
\usepackage{amsmath}
\usepackage{amssymb}
\usepackage{graphicx}
\usepackage{layout}
\usepackage{epsfig}
\usepackage{indentfirst}
\newcommand{\va}{{\mathbb \varepsilon}}
\newcommand{\R}{{\mathbb R}}

\newcommand{\X}{{\mathbb X}}

\newcommand{\N}{{\mathbb N}}

\newcommand{\La}{\Lambda}

\newcommand{\dis}{\displaystyle}
\theoremstyle{plain}
\newtheorem{theorem}{Theorem}[section]

\newtheorem{lemma}{Lemma}[section]
\newtheorem{proposition}{Proposition}[section]

\numberwithin{equation}{section}
\newcommand{\beq}{\begin{equation}}
\newcommand{\eeq}{\end{equation}}

\begin{document}

\def\ST{\songti\rm\relax}
\def\HT{\bf\relax}
\def\FS{\fangsong\rm\relax}
\def\KS{\kaishu\rm\relax}
\def\REF#1{\par\hangindent\parindent\indent\llap{#1\enspace}\ignorespaces}
\def\pd#1#2{\frac{\partial#1}{\partial#2}}
\def\ppd#1#2{\frac{\partial^2#1}{\partial#2^2}}

\title  {{ Large time behavior of  solutions for a Cauchy problem on nonlinear conservation laws with large initial data in the whole space}
\thanks {Research was supported by  Natural Science Foundation of China
(11101160,11271141),  Natural Science Foundation of Guangdong Province, China (2016A030313390) and  China Scholarship Council (201508440330)}}
\author
{\small { Lingyu Jin, Lang Li and Shaomei Fang}\\
\footnotesize  { Department of Mathematics,   South China Agricultural University,   Guangzhou 510642,China}}

\maketitle

 \noindent
 {\bf Abstract:}
We consider the Cauchy problem on a nonlinear conversation law with large initial data.  By Green's function methods, energy methods, Fourier analysis, frequency decomposition, pseudo-differential operators,  we obtain  the global existence and the optimal decay estimate of $t$.


\section{Introduction}
In this paper we consider the Cauchy problem on a scalar conservation law with a diffusion-type source
\begin{eqnarray}\label{1.1}
\begin{cases}\dis
u_t-\Delta P_{s_1}u=-\text{div}P_{s_2} f(u), \,\,&x\in\R^n ,t>0,\\
u(0,x)=u_0(x),\,\,&x\in\R^n.
\end{cases}
\end{eqnarray}
where $f(u)$ is a given smooth function, $P_{s_i},i=1,2$ is a pseudo-differential operator defined by
\begin{equation}
P_{s_i}u =\mathcal{F}^{-1}\{\frac{1}{(1+|\xi|^2)^{s_i}}\mathcal{F}u\}=\frac{1}{(I-\Delta)^{s_i}}u,
\end{equation}
then (\ref{1.1}) can be rewritten as 
\begin{eqnarray}\label{1.1j}
\begin{cases}\dis
u_t-\frac{\Delta}{(I-\Delta)^{s_1}}u=\frac{-\text{div} f(u^{})}{(I-\Delta)^{s_2}},\,\,&x\in\R^n ,t>0,\\
u(0,x)=u_0(x),\,\,&x\in\R^n.
\end{cases}
\end{eqnarray}
 
 Several mathematical models in different context can be provided some concrete examples related to equation (\ref{1.1}) for different assumptions on $s_1,s_2$. When $s_2=0$, (\ref{1.1})  becomes the following form
 \begin{eqnarray}\label{1.3}
 \begin{cases}\dis
 u_t-\frac{\Delta}{(I-\Delta)^{s_1}}u=-\text{div} f(u^{}), \,\,&x\in\R^n ,t>0,\\
 u(0,x)=u_0(x),\,\,&x\in\R^n.
 \end{cases}
 \end{eqnarray}
investigated in \cite{DR}. The authors studied the well-posedness and large time behavior for the Cauchy problem (\ref{1.3}) for  arbitrary space dimensions and for all $s_1\in \R$ in the framework of small-amplitude classical solutions. It is convenient to say that equation (\ref{1.3}) is of the regularity-gain type for $s_1<1$ whereas of the regularity-loss type for $s_1>1$. They established the time-decay rate of solutions and their derivatives up to some order, where the extra regularity on initial data was required in  the regularity-loss type $(s_1< 1$). In particular, when $s_1=1,s_2=0$  (\ref{1.1}) was derived in \cite{R} as the corresponding extension of the Navier-Stokes equations via the regularization of the Chapman-Enskog expansion from the Boltzmann equation, which is intended to obtain a bounded approximation of the linearized collision operator for both low and high frequencies.

When $s_2=s_1=1$, (\ref{1.1}) is connected with the famous BBM equation. BBM equation used as an alternative to the KdV equation which describes unidirectional propagation of weakly long dispersive waves \cite{BBM}. As a model that characterizes long waves in nonlinear dispersive media, the BBM equation, like the KdV equation, was formally derived to describe an approximation for surface water waves in a uniform channel. For the related study of the BBM equation, we can refer to \cite{K,S,SSW,XX}. Here we only give the results in \cite{K,XX}.

 Grzegorz Karch in \cite{K} considered the following problem in one dimension,
\begin{equation}\label{1.5}
\begin{cases}
u_t-u_{xxt}-\eta u_{xx}+bu_x=-(f(u))_x,\,\,&x\in\R ,t>0,\\
u(0,x)=u_0(x),\,\,&x\in\R,
\end{cases}
\end{equation}
they showed the decay in time of the spatial $L^p-$norm ($1\leq p\leq \infty$) of solutions under general assumptions about the nonlinearity. In particular, solutions of the nonlinear equation (\ref{1.5}) have the same long time behavior as their linearizations at $u=0$.
For the Cauchy problem in general $n$ dimensions, \cite{XX} studied the following problem:
\begin{equation}
\begin{cases}
u_t-\Delta
 \partial _t u-\eta\Delta u+(\beta\cdot \nabla ) u=-\text{div} f(u^{}),\,\,&x\in\R^n ,t>0,\\
u(0,x)=u_0(x),\,\,&x\in\R^n.
\end{cases}
\end{equation}
With small initial data the authors obtained the global existence and optimal $L^2-$norm convergence rates of the solutions.

 It is easy to notice that the existence, the decay in time and the regularity of the solution depend on $s_1,s_2$. In this paper, we obtain the 
 the global existence of the solution of (\ref{1.1}) for all $s_2\geq s_1$ with large initial datas. In addition, under the assumptions $s_2\geq s_1$ and $0\leq s_1<1$ we also have the optimal decay in time of the spatial $L^2-$norm of the solution with large initial data. Comparing with the work in \cite{DR,K}, we study the solution with large perturbation of initial data other than small initial data. The difficulty between large initial data and small initial data is totally different. In fact we  obtain the global existence and the optimal decay estimates by making use of Green's function method, Fourier analysis, frequency decomposition, pseudo-differential operators and the energy methods. Furthermore, the difficulty of (\ref{1.1}) lies in that there is no maximum principle like other equations such as viscous Burgers equation. Then we make $L^1$ and $L^\infty$ estimations of the solutions which are the key steps for the proof of Theorem 1.1 and Theorem 1.2.

We now introduce some notations. 

 In what follows, we denote generic positive constants by $c$ and $C$ which may change from line to line. The Fourier transform
 $\hat{f}$  of a tempered distribution $f(x)$ on $\R^n$ is defined as
 \[\mathcal{F}[f(x)](\xi)= {\hat f(\xi)}=\frac{1}{(2\pi)^n
 	}\int_{R^n}f(x)e^{-i\xi\cdot x}d\xi.\]

We will denote the square root of the Laplacian $(-\Delta)^{\frac{1}{2}}$ by $\Lambda$ and obviously
\[\widehat{\Lambda f}=|\xi|\hat{f}(\xi).\]
Let $H^{s}(\R^n)$ be the general fractional Sobolev space with the norm $$\|f\|^2_{H^s}=\int_{\R^n}(1+|\xi|^2)^{s}|\hat f|^2d\xi.$$

For $s=0$, $H^0(\R^n)=L^2(\R^n)$.
The space $\dot{H}^s(\R^n)$ denotes the homogeneous fractional Sobolev space with the norm
 $$\|f\|_{\dot{H}^s}^2=\int_{\R^n}|\xi|^{2s}|\hat f|^2d\xi.$$
 It can be easily deduced that there exist constants $c_0,c_1>0$ such that
 \begin{equation}\label{1.2}
c_0(\|f\|_{L^2}^2+\|f\|_{\dot{H}^s}^2) \leq \int_{\R^n}(1+|\xi|^2)^{s}|\hat f|^2d\xi \leq c_1 (\|f\|_{L^2}^2+\|f\|_{\dot{H}^s}^2),
 \end{equation}
 which follows $(\|f\|_{L^2}^2+\|f\|_{\dot{H}^s}^2)^{1/2}$ be an equivalent norm on $H^s(\R^n)$.

In this paper, we assume that
\begin{equation}\label{1.8}
	f(x,u)=u^{\theta+1},  1\leq \theta\leq \theta_0,\theta \in \N
\end{equation}
where 
\begin{equation}\label{1.9}
\theta_0=\begin{cases}\infty, \text{ if } n\leq 2s_2,\\
\frac{2\bigl(1+2(s_2-s_1)\bigl)}{n-2s_2}, \text{ if } n>2s_2,
\end{cases}
\end{equation}
 we are mostly interested in the existence and the large time behaviour of the solutions for the Cauchy problem of (\ref{1.1}).

 Now we introduce the main theorems in this paper.
  \begin{theorem}
  	Let $s> \{s_2,\frac{n}{2}\}$, $s_2>s_1$. Assume that $u_0\in H^s(\R^n)\bigcap L^1(\R^n)$, then  (1.1)  has a global solution  
  	\[u\in L^\infty(0,\infty;H^s(\R^n)).\]
   \end{theorem}
  
 \begin{theorem}
Assume $n>2, u_0\in L^1(\R^n)\bigcap H^s(\R^n)$, $s> \max\{s_2,n/2\}, s_2>s_1, 0\leq s_1< 1$ and $u\in L^\infty(0,\infty;H^s(\R^n))$ is the solution of   (1.1)  with initial data $u_0$, then
 $$\|\Lambda ^s  u(t)\|_{L^2} \leq c (1+t)^{-\frac{n}{4}-\frac{s}{2}}.$$
 \end{theorem}
An extra  assumption   $0\leq s_1< 1$  is imposed in Theorem 1.2.   In fact,  if $s_1\geq 1$ the linear part of (\ref{1.1}) would be a regularity-loss type as discussed in \cite{DR}, \cite{JLF}.  For the case $s_1\geq 1$ so far we can only obtain the decay estimation with small initial data.

The rest of paper is organized as follows. In Section 2, we give some preliminary lemmas. In Section 3, we establish some decay estimates for Green's function related to (\ref{1.1j}). In Section 4, we get the local existence for the solution of (\ref{1.1}). Finally,  we prove our main results Theorem 1.1 and Theorem 1.2 in Section 5.

\section{Preliminaries}
In this section, we give some preliminary lemmas. 
\begin{lemma}\label{l2.1}
	(Refer to \cite{S2}, \cite{J}) Assume 
	\[g\in W^{l,p_2}(\R^n)\bigcap L^{p_1}(\R^n) \text{ and }h\in W^{l,q_1}(\R^n)\bigcap L^{q_2}(\R^n), 1\leq p_1,p_2,q_1,q_2\leq  \infty,l>0.\] Then there exists a constant $C>0$ such that 
\begin{eqnarray}
\|\La ^l(gh)\|_{L^r}\leq C(\|g\|_{L^{p_1}}\|\La^l h\|_{L^{q_1}}+\|\La^l g\|_{L^{p_2}}\|h\|_{L^{q_2}}),
\end{eqnarray}
where $\frac{1}{r}=\frac{1}{p_1}+\frac{1}{q_1}=\frac{1}{p_2}+\frac{1}{q_2}.$
	\end{lemma}
	\begin{lemma}\label{l2.9} Assume $u\in W^{l,p}(\R^n)\bigcap L^{q\theta}(\R^n)$.
	For all $\theta,l>0,\theta\in \mathbb{Z}$	\begin{eqnarray}\label{2.3j}
		\|\La^l (u^{\theta+1})\|_{L^r}\leq C\|\Lambda^lu\|_{L^p}(\|u^\theta\|_{L^q}+\|u^{\theta-1}\|_{L^q}\|u\|_{L^q}+\cdots+\|u\|^\theta_{L^q})
		\end{eqnarray}
		where $\frac{1}{p}+\frac{1}{q}=\frac{1}{r}$.
	\end{lemma}
	\begin{proof}
		From Lemma \ref{l2.1}, we have (\ref{2.3j}) for $\theta=1$, a. e.,
		\begin{eqnarray}
		\|\La^l (u^{2})\|_{L^r}\leq C\|\Lambda^lu\|_{L^p}\|u\|_{L^q}.
		\end{eqnarray}
	Assume that (\ref{2.3j}) holds  for $\theta=k$, then
		\begin{eqnarray}
		\|\La ^l(u^ku)\|_{L^r}\leq c(\|u^{k}\|_{L^q}+\|u^{k-1}\|_{L^q}\|u\|_{L^q}+\cdots+\|u\|^k_{L^q})\|\Lambda^l u\|_{L^p}.
		\end{eqnarray}
		Thus for $\theta=k+1$, \begin{equation}
		\begin{split}
		\|\La ^l(u^{k+1}u)\|_{L^r}&\leq C(\|u^{k+1}\|_{L^q}\|\La^l u\|_{L^p}+\|\Lambda^l u^{k+1}\|_{L^p}\|u\|_{L^q})
		\\&\leq c(\|u^{k+1}\|_{L^q}+\|u^k\|_{L^q}\|u\|_{L^q}+\|u^{k-1}\|_{L^q}\|u\|^2_{L^q}+\cdots+\|u\|^{k+1}_{L^q})\|\Lambda^l u\|_{L^p}.
		\end{split}
		\end{equation}
		Using an induction argument, we have (\ref{2.3j})  for the general case.
		\end{proof}
		Lemma \ref{l2.1} and Lemma \ref{2.3j} are used to deal with the estimation for the non-linear term $u^{\theta+1}$.  In the following we give Galiardo-Nirenberg inequality which is known as the interpolation inequality.
\begin{lemma}\label{l2.4}
(Galiardo-Nirenberg inequality, \cite{T1})
Suppose that $u\in L^q(\mathbb R^n)\cap W^{m,r}(\mathbb R^n)$, where $1\leq q, r\leq\infty$. Then there exists a constant $C>0$, such that
\begin{eqnarray}\|D^j	u\|_{L^p}\leq C\|D^mu\|^a_{L^r}\|u\|^{1-a}_{L^q},
\end{eqnarray}
where
$$\frac{1}{p}=\frac{j}{n}+a\left(\frac{1}{r}-\frac{m}{n}\right)+(1-a)\frac{1}{q},$$
$1\leq p\leq\infty$, $j$ is an integer, $0\leq j\leq m, j/m\leq a\leq 1$. If $m-j-n/r$ is a nonnegative integer, then the inequality holds for $j/m\leq a<1$.
\end{lemma}
Next we will give a lemma which plays important roles in  making estimates on the low frequency part of the Green's function in Section 3.
\begin{lemma}\label{l3.1}
	({Lemma 3.1 in \cite{WY}})
	If $\hat f(t,\xi)$ has compact support in the variable $\xi,N$ is a positive integer, and there exists a constant $b>0$, such that $\hat f(t,x)$ satisfies 
	\begin{equation}
	|D^\beta (\xi^\alpha\hat f(t,\xi))|\leq c(|\xi|^{(|\alpha|+k-|\beta|)_+}+|\xi|^{|\alpha|+k}t^{\frac{|\beta|}{2}})(1+t|\xi|^{2})^me^{-b|\xi|^{2}t},
	\end{equation}
	for any two multi-indexes $\alpha,\beta$ with $|\beta|\leq 2N$, then
	\begin{equation}
	|D^\alpha_x f(t,x)|\leq C_Nt^{-\frac{(n+|\alpha|+k)}{2}}B_N(t,|x|),
	\end{equation}
	where $k, m\in \N$, $(u)_+=\max\{0,u\}$, $B_N(t,|x|)=(1+\frac{|x|^{2}}{(1+t)})^{-N}$ and $N$ is any positive integer.
\end{lemma}

Finally we give an inequality which can also be thought as a interpolation inequality.
\begin{lemma}\label{l2.5}
	Let $D\subset \R^n$, for all $r_2\geq r_1>0$.
	\begin{equation}\label{2.9}
	\int_{D}( |\xi|^{r_1} \hat u)^2d\xi\leq (\int_{D}(|\xi|^{r_2} \hat u)^2d\xi)^a( \int_{D}\hat u^2d\xi) ^{1-a}
	\end{equation}
where $a=\frac{r_1}{r_2}$.
\end{lemma}
\begin{proof}
By H\"older inequality, it gives
\begin{equation}
\int_{D}( |\xi|^{r_1} \hat u)^2d\xi\leq \Bigl (\int_{D}( |\xi|^{2r_1}\hat u^{\frac{2r_1}{r_2}})^{\frac{r_2}{r_1}}d\xi\Bigl )^{\frac{r_1}{r_2}} (\int_{D}( \hat u^{2-\frac{2r_1}{r_2}})^{\frac{r_2}{r_2-2r_1}}d\xi\Bigl )^{1-\frac{r_1}{r_2}}.
\end{equation}
Lemma \ref{l2.5} is complete.
\end{proof}
\section{Decay estimates of Green's functions associated with (\ref{1.1})}
In this section, we study the decay property of the Green's function associated with (\ref{1.1}), which satisfies
\begin{eqnarray}\label{2.1}
\begin{cases}
\dis (\partial_t-\frac{\Delta}{(I-\Delta)^{s_1}})G(t,x)=0,\,\,& x\in\R^n,\,t>0,\\
G(0,x)=\delta(x),\,\,&x\in\R^n,
\end{cases}
\end{eqnarray}
where $\delta(x) $ is the Dirac function and $0\leq s_1<1$. Applying the Fourier transform with respect to $x$,  we arrive at the expression \[\hat G(t,\xi)=e^{-\frac{|\xi|^{2}}{(1+|\xi|^2)^{s_1}}t}.\]
Denote  \[G(x,t)=\mathcal{F}^{-1}e^{-\frac{|\xi|^{2}}{(1+|\xi|^2)^{s_1}}t}.\]
So the solution of the Cauchy problem of (\ref{1.1}) has the following integral representation
\begin{equation}\label{3.2}
u=G*u_0-\int^t_0 G(t-\tau,\cdot)*\frac{\text{div}f(u)}{(I-\Delta)^{s_2}}d\tau.
\end{equation}
 We are going to obtain the decay estimates of $\|G(x,t)\|_{L^1} $ and $\|\nabla G(t,x)\|_{L^1}$.  Then we first need to get the point wise estimates of $G(t,x)$. 
Let \begin{eqnarray}\label{j3}
\chi_1(\xi)=\begin{cases}
1,&|\xi|\leq \delta,\\ 
0,&|\xi|\geq 2\delta ;
\end{cases}\ \ \ \ \chi_3(\xi)=\begin{cases}
1,&|\xi|\geq R,\\ 
0,&|\xi|\leq  R-1 ;
\end{cases}
\end{eqnarray}
be the smooth cut-off functions with $R>2,\delta>0$. Set 
\[\chi_2(\xi)=1-\chi_1(\xi)-\chi_3(\xi),\] and \begin{equation}\label{3.4jj}
\hat G_i(\xi,t)=\chi_i(\xi)\hat G(\xi,t),\,\, i=1,2,3.
\end{equation}


Thanks to Lemma \ref{l3.1}, we can get the following estimate for $G_1(t,x)$.
\begin{proposition}\label{p1}
	For sufficiently small $\delta$, there exists a constant $c>0$ such that \begin{equation}
	|D^\alpha G_1(t,x)|\leq c t^{-(n+|\alpha|)/2}B_N(t,|x|)
	\end{equation}
	where  $\alpha$ is a multi-index.
\end{proposition}
\begin{proof}
	For $|\xi|$ being sufficiently small, according to the Taylor expansion we have that\begin{equation}\label{3.8}
	\frac{|\xi|^{2}}{(1+|\xi|^2)^{s_1}}=|\xi|^2(1+O(|\xi|^2))=|\xi|^2+O(|\xi|^4).
	\end{equation}
	Therefore
	\begin{equation}
	\hat G(t,\xi)=e^{-|\xi|^2t}(1+O(|\xi|^4)t).
	\end{equation}
	Noticing that $\hat G(t,\xi)$ is a smooth function to variable $\xi$ near $|\xi|=0$, we get that when $|\beta|\leq 2N$
	\begin{equation}
	|D^\beta_{\xi}(\xi^\alpha \hat G)|\leq c(|\xi|^{(|\alpha|-|\beta|)_+}+|\xi|^{|\alpha|}t^{|\beta|/2})(1+t|\xi|^2)^{|\beta|/2+1}e^{-|\xi|^2t},\text{ for all } |\xi|\leq 2\delta.
	\end{equation}
	It gives that 
	\begin{equation}
	|D^\beta_\xi (\xi^\alpha \hat G_1)|\leq c(|\xi|^{(|\alpha|-|\beta|)_+}+|\xi|^{|\alpha|}t^{|\beta|/2})(1+t|\xi|^2)^{|\beta|/2+1}e^{-|\xi|^2t}.
	\end{equation}
	From Lemma \ref{l3.1}, we have
	\[|D^\alpha G_1(t,x)|\leq ct^{-(n+|\alpha|)/2}B_N(t,|x|).\]
\end{proof}
Next we come to consider the estimation for $G_2(x,t)$.
\begin{proposition}\label{p2}
For fixed $0<\delta<1$ and $R>2$, there exist positive constants $m_0$ and $c$ such that \begin{equation}
|D^\alpha G_2(t,x)|\leq ce^{-\frac{t}{2m_0}}B_N(t,|x|),
\end{equation}
where  $\alpha$ is a multi-index.
\end{proposition}
\begin{proof}
	For any fixed $\delta<1$,  choosing $m$ sufficiently large such that $m>\frac{1}{\delta^2}$, we obtain
	\begin{equation}
\frac{-|\xi|^2}{(1+|\xi|^2)^{s_1}}\leq -\frac{1}{2m}, \text{ for }|\xi|\geq \delta
	\end{equation}
which implies that
\begin{equation}\label{3.14}
|\hat G_2(t,\xi)|=|\chi_2(\xi)\hat G(t,\xi)|=|\chi_2(\xi)e^{\frac{-|\xi|^2}{(1+|\xi|^2)^{s_1}}t}|\leq ce^{-\frac{1}{2m}t}.
\end{equation}	
Then
\begin{equation}
|G_2(t,x)|\leq c|\int_{\R^n}e^{ix\xi}\hat G_2(t,\xi)d\xi |\leq ce^{-\frac{1}{2m}t}.
\end{equation}
Now we shall give an estimate to $x^\beta G_2(t,x)$ by induction on $\beta$. We claim that  for any multi-index $\beta$,
\begin{equation}\label{3.16}
|D^\beta _\xi \hat G_2(t,\xi)|\leq c(1+t)^{|\beta|} e^{-\frac{t}{2m}}.
\end{equation}
In fact for $ |\beta|=0$, (\ref{3.16}) follows from (\ref{3.14}).
Assume that $|\beta|\leq j-1$, (\ref{3.16}) is satisfied. 
Performing the Fourier transform of (\ref{2.1}) with respect to $x$, it follows
\begin{equation}
\begin{cases}
&\partial _t \hat G(t,\xi)+\frac{|\xi|^2}{(1+|\xi|^2)^{s_1}}\hat G(t,\xi)=0, t>0, \\
&\hat G(0,\xi)=1,
\end{cases}
\end{equation}
then multiplying with $\chi_2(\xi)$, we have
\begin{equation}\label{3.18}
\begin{cases}
&\partial _t \hat G_2(t,\xi)+\frac{|\xi|^2}{(1+|\xi|^2)^{s_1}}\hat G_2(t,\xi)=0, t>0, \\
&\hat G_2(0,\xi)=\chi_2(\xi).
\end{cases}
\end{equation}
Let the operator $D^\beta_\xi$ act on (\ref{3.18}), then we have
\begin{equation}\label{3.19}
\begin{cases}
&\partial _t D^\beta_\xi\hat G_2(t,\xi)+\frac{|\xi|^2}{(1+|\xi|^2)^{s_1}}D^\beta_\xi\hat G_2(t,\xi)=-F(\xi), t>0, \\
&D^\beta_\xi\hat G_2(0,\xi)=a_0.
\end{cases}
\end{equation}
where $a_0$ is a polynomial of $|\xi|$ and 
\begin{equation}
F(\xi)=\sum_{\beta_1+\beta_2=\beta, |\beta_1|\neq 0}\frac{|\beta|!}{|\beta_1|!|\beta_2|!}D^{\beta_1}_\xi(\frac{|\xi|^2}{(1+|\xi|^2)^{s_1}})D^{\beta_2}_\xi\hat G_2(t,
\xi).
\end{equation}
Obviously, 
\begin{equation}\label{3.21j}
|F(\xi)|\leq c|D^{\beta_2}_{\xi}\hat G_2(t,\xi)|\leq c(1+t)^{|\beta_2|} e^{-\frac{t}{2m}}\text{ for all }|\beta_2|\leq |\beta|-1.
\end{equation}
Hence for $|\beta|=j$, 
\begin{equation}
D^{\beta}_\xi\hat G_2(t,\xi)=a_0\hat G_2(t,\xi)-\int^t_0 \hat G_2(t-s,\xi)F(\xi)ds.
\end{equation}
Combing this with (\ref{3.21j}),
 we have 
\begin{equation}
\begin{split}
|D^{\beta}_\xi\hat G_2(t,\xi)|&\leq c e^{-\frac{t}{2m}}+c\int^t_0 e^{-\frac{t-s}{2m}}(1+s
)^{|\beta|-1}e^{-s/2m}ds\\
&\leq c(1+t)^{|\beta|}e^{-\frac{t}{2m}}
\end{split}
\end{equation}
which implies that (\ref{3.16}) is valid for $|\beta|=j$. 
Then for any $\beta$, we have
\begin{equation}\label{3.24}
\begin{split}
|x^\beta D^\alpha_xG_2(t,x)|&\leq c|\int_{\R^n}e^{ix\xi}D^\beta_\xi (\xi^\alpha\hat G_2(t,\xi))d\xi|\\
&\leq e^{-\frac{t}{2m}}\int_{\va\leq |\xi|\leq R}(|\xi|^{|\alpha|}+|\xi|^{||\alpha|-|\beta||})(1+t)^{|\beta|}d\xi\\
&\leq c(1+t)^{|\beta|}e^{-\frac{t}{2m}}\\
&\leq c(1+t)^{|\beta|/2}e^{-\frac{t}{2m_0}}
\end{split}
\end{equation}
for some $m_0>m$. This follows that for $\beta=0$
\begin{equation}
\label{3.25}
\begin{split}
|D^\alpha_xG_2(t,x)|
\leq ce^{-\frac{t}{2m}}
\end{split}\end{equation}
and taking $\beta=\bar\beta_i=2Ne_i,i=1,2,\cdots n$ ($e_i$ is the ith unite coordinate vector) in (\ref{3.24}), it follows
\begin{equation}\label{3.26}
\dis
|D^\alpha_xG_2(t,x)|\leq c\frac{(1+t)^{N}}{\dis\sum_{i=1}^{n}x_i^{2N}}e^{-\frac{t}{2m_0}}\leq c\frac{(1+t)^{N}}{|x|^{2N}}e^{-\frac{t}{2m_0}}\text{ for all } |x|^2\geq (1+t).\end{equation}
then from (\ref{3.25}) and (\ref{3.26}) that
\begin{equation}\label{3.27j}
|D^\alpha_xG_2(t,x)|\leq c\min\{\frac{(1+t)^{N}}{|x|^{2N}},1\}e^{-\frac{t}{2m_0}}.
\end{equation}
Since 
\begin{equation}
1+\frac{|x|^2}{1+t}\leq \begin{cases}
2,&|x|^2\leq 1+t,\\
2\frac{|x|^2}{1+t},&|x|^2\geq 1+t,
\end{cases}
\end{equation}
we have 
\begin{equation}\label{3.29jj}
\min\{\frac{(1+t)^{N}}{|x|^{2N}},1\}\leq 2^N (1+\frac{|x|^2}{(1+t)})^{-N}\leq c B_N(t,|x|),
\end{equation}
which implies 
\begin{equation}\label{3.29j}
|D^\alpha_xG_2(t,x)|
\leq ce^{-\frac{t}{2m}}B_N(t,|x|).
\end{equation}
\end{proof}
\begin{proposition}\label{p3}
		For sufficiently large $R$, there exists a positive constant $c$ such that
	\begin{equation}
	|D^\alpha G_3(t,x)|\leq ct^{\frac{-(n+|\alpha|)}{{2\nu}}}B^{\nu}_N(t,|x|)
\text{ for all } t\geq 1	\end{equation}
	where $B_N^\nu(t,|x|)=(1+\frac{|x|^{2\nu}}{(1+t)})^{-N}$, $\nu=1-s_1$ and $\alpha$ is a multi-index with $|\alpha|\leq 1$.
\end{proposition}
\begin{proof}
	Firstly we  give an estimate of $D_\xi^\beta(\xi^\alpha \hat G_3(t,\xi))$.

	For $|\xi|$ sufficiently large, according to the Taylor expansion we have that\begin{equation}\label{3.31}
	\frac{|\xi|^{2}}{(1+|\xi|^2)^{s_1}}=\frac{|\xi|^{2-2s_1}}{(1+\frac{1}{|\xi|^2})^{s_1}}
	=|\xi|^{2-2s_1}(1+O(\frac{1}{|\xi|^2}))=|\xi|^{2-2s_1}+O(\frac{1}{|\xi|^{2s_1}}),
	\end{equation}
it follows for $|\xi|$ large enough,
	\begin{equation}
	|\hat G(t,\xi)|\leq ce^{-|\xi|^{2-2s_1}t}.
	\end{equation}

	 We claim that  for any multi-index $\beta$,
	 for $\nu=1-s_1\geq \frac{1}{2}$,
	\begin{equation}\label{3.33}
	|D^\beta _\xi \hat G_3(t,\xi)|\leq ce^{-|\xi|^{2\nu}t}(|\xi|^{2\nu-1}t)^{|\beta|},\forall t\geq 1;
	\end{equation}
	 for $\nu=1-s_1< \frac{1}{2}$,
	 \begin{equation}\label{3.33j}
	 |D^\beta _\xi \hat G_3(t,\xi)|\leq ce^{-|\xi|^{2\nu}t}t^{|\beta|},\forall t\geq 1.
	 \end{equation}
	 For the proof of (\ref{3.33}), it is sufficient to prove 
   \begin{equation}\label{3.34}
   |D^\beta _\xi e^{-|\xi|^{2\nu}t}|\leq e^{-|\xi|^{2\nu}t}(|\xi|^{2\nu-1}t)^{|\beta|}, \text{ for all } |\xi|\geq R-1\geq 1,t\geq 1.
   \end{equation}
	In fact for $ |\beta|=0$, obviously (\ref{3.33}) follows.
	By induction method, we assume (\ref{3.33}) is satisfied for all $|\beta|\leq j$. Then for $|\beta|=j+1$, let $\bar \beta $ be any multi-index with $|\bar \beta|=j$, since,  for $ i=1,\cdots,n$, $|\xi|\geq 1$,
	\begin{equation}\label{3.36}
	\begin{split}
	D_\xi^\beta(e^{-|\xi|^{2\nu}t})&=D_\xi^{\bar\beta}(\partial_{\xi_i}(e^{-|\xi|^{2\nu}t}))\\
	&=D_\xi^{\bar\beta}\bigl(e^{-|\xi|^{2\nu}t}(-\xi_i)|\xi|^{2\nu-2}t\bigl)\\
	&\leq c\sum_{|\beta_1|+|\beta_2|=j}D_\xi^{\beta_1}\bigl(e^{-|\xi|^{2\nu}t}\bigl)D_\xi^{\beta_2}\bigl((-\xi_i)|\xi|^{2\nu-2}t\bigl)
\\&	\leq ce^{-|\xi|^{2\nu}t}(|\xi|^{2\nu-1}t)^{j}(|\xi|^{2\nu-1}t)
	=ce^{-|\xi|^{2\nu}t}(|\xi|^{2\nu-1}t)^{j+1} 
	\end{split}
	\end{equation}
	where the last inequality but one  holds because 
	\begin{equation}\label{3.37}
	D_\xi^{\beta_2}\bigl(-\xi_i|\xi|^{2\nu-2}t\bigl)\leq c|\xi|^{2\nu-1}t \text{ for any $|\xi|\geq 1,t\geq 1$.}
	\end{equation}
	Then (\ref{3.34}) follows for $|\beta|=j+1$. The claim (\ref{3.33}) proved for $\nu\geq 1/2$.
	Similar to (\ref{3.34})-(\ref{3.37}), we can obtain (\ref{3.33j}).
	The claim is proved.
	
	For $|\alpha|\leq 1,|\alpha|\leq |\beta|$, since
	\begin{equation}
|	D_\xi^\beta(\xi^\alpha e^{-|\xi|^{2\nu}t})|\leq c|\xi^\alpha 	D_\xi^\beta(e^{-|\xi|^{2\nu}t})|+c|D_\xi^{\bar\beta}(e^{-|\xi|^{2\nu}t})|,
	\end{equation}
	where $\bar\beta$ is a multi-index with $|\bar{\beta}|=|\beta|-1$,
from
(\ref{3.33}) and (\ref{3.33j}), we have	
  for $\nu=1-s_1\geq \frac{1}{2}$,
  \begin{equation}\label{3.39}
  |D^\beta_\xi \bigl(\xi^\alpha\hat G_3(t,\xi)\bigl)|\leq c|\xi|^{|\alpha|} e^{-|\xi|^{2\nu}t}(|\xi|^{2\nu-1}t)^{|\beta|},\forall t\geq 1;
  \end{equation}
  for $\nu=1-s_1< \frac{1}{2}$,
  \begin{equation}\label{3.40}
  |D^\beta_\xi \bigl (\xi^\alpha \hat G_3(t,\xi)\bigl)|\leq c|\xi|^{|\alpha|} e^{-|\xi|^{2\nu}t}t^{|\beta|},\forall t\geq 1.
  \end{equation}
  
 Secondly, we will establish the estimation for $|x^\beta G_3(x,t)|$.
 For $\beta=0$, 
 \begin{equation}\label{3.41}
 \begin{split}
 |D^\alpha_xG_3(t,x)|&\leq c|\int_{\R^n}e^{ix\xi} (\xi^\alpha\hat G_3(t,\xi))d\xi|\\
 &\leq c\int_{|\xi|\geq R-1}e^{-|\xi|^{2\nu}t}|\xi|^\alpha d\xi\\
 &\leq ct^{-\frac{n+|\alpha|}{2\nu}}\int_{\R^n}e^{-|\eta|^{2\nu}}|\eta|^\alpha d\eta \\
 &\leq ct^{-\frac{n+|\alpha|}{2\nu}}.
 \end{split}
 \end{equation}
 For $|\beta|\geq |\alpha|,\nu=1-s_1\geq 1/2$, from (\ref{3.39}),
  \begin{equation}\label{3.42}
  \begin{split}
  |x^\beta D^\alpha_xG_3(t,x)|&\leq c|\int_{\R^n}e^{ix\xi} D^\beta_{\xi}\bigl (\xi^\alpha\hat G_3(t,\xi)\bigl )d\xi|\\
  &\leq c\int_{|\xi|\geq R-1}|\xi|^{|\alpha|}e^{-|\xi|^{2\nu}t}(|\xi|^{2\nu-1}t)^{|\beta|} d\xi\\
  &\leq ct^{-\frac{n+|\alpha|}{2\nu}}\int_{\R^n}e^{-|\eta|^{2\nu}}|\eta|^\alpha |\eta|^{({2\nu}-1)\beta}t^{\frac{|\beta|}{2\nu}} d\eta\\
  &\leq ct^{-\frac{n+|\alpha|}{2\nu}}(1+t)^{\frac{|\beta|}{2\nu}}.
  \end{split}
  \end{equation}
   For $|\beta|\geq |\alpha|, \nu=1-s_1<1/2$, from (\ref{3.40}),
   \begin{equation}\label{3.43}
   \begin{split}
   |x^\beta D^\alpha_xG_3(t,x)|&\leq c|\int_{\R^n}e^{ix\xi} D^\beta_{\xi}\bigl (\xi^\alpha\hat G_3(t,\xi)\bigl )d\xi|\\
   &\leq c\int_{|\xi|\geq R-1}|\xi|^{|\alpha|}e^{-|\xi|^{2\nu}t}t^{|\beta|} d\xi\\
   &\leq ct^{-\frac{n+|\alpha|}{2\nu}}\int_{\R^n}e^{-|\eta|^{2\nu}}|\eta|^\alpha t^{|\beta|} d\eta \\
  & \leq  ct^{-\frac{n+|\alpha|}{2\nu}}(1+t)^{|\beta|}
   \leq ct^{-\frac{n+|\alpha|}{2\nu}}(1+t)^{\frac{|\beta|}{2\nu}}.
   \end{split}
   \end{equation}
	Letting $|\beta|=2N$, from (\ref{3.42}) and (\ref{3.43}), similar to (\ref{3.26}), we have, 
	\begin{equation}\label{3.44}
	\dis
	|D^\alpha_xG_3(t,x)|\leq c\frac{(1+t)^{N/\nu}}{\dis\sum_{i=1}^{n}x_i^{2N}}t^{-\frac{n+|\alpha|}{2\nu}}\leq c\frac{(1+t)^{N/\nu}}{|x|^{2N}}t^{-\frac{n+|\alpha|}{2\nu}}\text{ for all } |x|^{2\nu}\geq (1+t).\end{equation}
	then from (\ref{3.41}) and (\ref{3.44}) that
	\begin{equation}
	\label{3.27}
	|D^\alpha_xG_3(t,x)|\leq c\min\{\frac{(1+t)^{N/\nu}}{|x|^{2N}},1\}t^{-\frac{n+|\alpha|}{2\nu}}.
	\end{equation}
	
	Since 
	\begin{equation}
	1+\frac{|x|^{2\nu}}{1+t}\leq \begin{cases}
	2,&|x|^{2\nu}\leq 1+t,\\
	2\frac{|x|^{2\nu}}{1+t},&|x|^{2\nu}\geq 1+t,
	\end{cases}
	\end{equation}
	we have 
	\begin{equation}\label{3.29}
	|D^\alpha_xG_3(t,x)|
	\leq ct^{-\frac{n+|\alpha|}{2\nu}}B_N^{\nu}(t,|x|)
	\end{equation}
where	$B_N^\nu(t,|x|)=(1+\frac{|x|^{2\nu}}{(1+t)})^{-N}$ and $\nu=1-s_1$.	
\end{proof}
In summary, we have the following theorem on the Green's function.
\begin{theorem}\label{T3.3}
	Assume $0\leq s_1< 1$. For $t\geq 1$
	\begin{equation}\label{3.48}
	\begin{split}
| G(t,x)|&\leq ct^{-\frac{n}{2}}B_N(t,|x|)+ct^{-\frac{n}{2(1-s_1)}}B^{1-s_1}_N(t,|x|),\\
|\nabla G(t,x)|&\leq ct^{-\frac{n+1}{2}}B_N(t,|x|)+ct^{-\frac{n+1}{2(1-s_1)}}B^{1-s_1}_N(t,|x|).
\end{split}
	\end{equation}

	Furthermore, for $t\geq 1$, 
\begin{equation}
\|G(t,x)\|_{L^1}\leq c, \,\,\|\nabla G(t,x)\|_{L^1}\leq ct^{-1/2};
\end{equation}
\begin{equation}
\|G(t,x)\|_{L^2}\leq ct^{-\frac{n}{4}}.
\end{equation}
\end{theorem}
\begin{proof}
	(\ref{3.48}) follows from  Proposition \ref{p1}-Proposition \ref{p2}
	 directly.
	 And from (\ref{3.48}), we have
	 \begin{equation}
	 \begin{split}
	 \|G(x,t)\|_{L^1}&\leq ct^{-\frac{n}{2}} \int_{\R^n}B_N(t,|x|)dx+ct^{-\frac{n}{2(1-s_1)}}\int_{\R^n}B^{1-s_1}_N(t,|x|)dx\\
	 & \leq ct^{-\frac{n}{2}} \int_{\R^n} (1+\frac{|x|^{2}}{(1+t)})^{-N}dx+ct^{-\frac{n}{2(1-s_1)}}\int_{\R^n} (1+\frac{|x|^{2(1-s_1)}}{(1+t)})^{-N}dx	\\
	 &\leq c \int_{\R^n} (1+|y|^2)^{-N}dy+c\int_{\R^n} (1+{|y|^{2(1-s_1)}})^{-N}dy\leq c	
	 \end{split}
	 \end{equation}
	 where the last inequality holds by choosing $2(1-s_1) N>n$.
	 Similarly we have 
	\[\|\nabla G(t,x)\|_{L^1}\leq c(t^{-1/2}+t^{t^{-\frac{1}{2(1-s_1)}}})\leq c t^{-1/2}. \]
	
	Furthermore， for $t\geq 1$
	\begin{equation}
	\begin{split}
	\|G(x,t)\|_{L^2}&\leq ct^{-\frac{n}{2}} \|B_N(t,|x|)\|_{L^2}+ct^{-\frac{n}{2(1-s_1)}}\|B^{1-s_1}_N(t,|x|)\|_{L^2}\\
	& \leq ct^{-\frac{n}{2}} (\int_{\R^n} (1+\frac{|x|^{2}}{(1+t)})^{-2N}dx)^{1/2}+ct^{-\frac{n}{2(1-s_1)}}(\int_{\R^n} (1+\frac{|x|^{2(1-s_1)}}{(1+t)})^{-2N}dx)^{1/2}	\\
	&\leq c t^{-\frac{n}{4}}(\int_{\R^n} (1+|y|^2)^{-2N}dy)^{1/2}+ct^{-\frac{n}{4(1-s_1)}}(\int_{\R^n} (1+{|y|^{2(1-s_1)}})^{-2N}dy)^{1/2}\\&\leq ct^{-\frac{n}{4}}	
	\end{split}
	\end{equation}
	where the last inequality holds by choosing $4(1-s_1) N>n$ .
\end{proof}
\section{Local existence}

In this section, we will construct a convergent sequence to get the local solution. 

Construct a sequence $\{u^m(t,x)\}$ which satisfies the following linear problem
\begin{eqnarray}
\dis
 & \dis \frac{\partial u^{m+1}}{\partial
	 t}-\frac{\Delta}{(I-\Delta)^{s_1}}u^{m+1}=\frac{-\text{div} （\bigl((u^{m})^{\theta+1}\bigl)}{(I-\Delta)^{s_2}}, \label{4.1}\\
&u^{m+1}(0,x)=u_0(x),\label{4.1j}
\end{eqnarray}
for $m\geq 1$ and $u^0(t,x)=0,u_0(x)\in H^s(\R^n)$. We will try to prove that the sequence is convergent in a space we construct and the limit is  the solution of the Cauchy problem (\ref{1.1}).

First we introduce a set of functions as follows. For a given integer
$s>\frac{n}{2}$(where $n$ is the spatial dimension)
\[\mathbb{X}=\{u(t,x)\ \mid \ \|u\|_{\X}<E\},\]
where $\dis\|u\|_{\X}=\sup_{0\leq t\leq T_0}\|u\|_{H^s(\R^n)}$, $T_0>0$ will be determined later and $E=c_0\|u_0\|_{H^s(\R^n)}$, $c_0\geq 4\sqrt{c_1}$ is a positive constant.  The metric in $\X$ is induced by the norm $\|u\|_{\X}$:
\[\rho(u,v)=\|u-v\|_{\X},\forall u,v\in \X.\]
Obviously $\X$ is a non-empty and complete metric space.

\begin{proposition}\label{a}
	Assume that $ 2s_2\geq s_1,s>\frac{n}{2}$.
	There exists some constant $T_0$ sufficiently small such that \[u^m(t,x) \in \X, \ \forall\, m\geq 1.\]\end{proposition}
\begin{proof}
	For $m=0$, we have
	\begin{equation}\label{j3.1}
u^1_t-\frac{\Delta}{(I-\Delta)^{s_1}}u^1=\frac{-\text{div}\bigl( (u^{0})^{\theta+1}\bigl)}{(I-\Delta)^{s_2}}\equiv 0.
	\end{equation}
	Multiplying (\ref{j3.1}) by $\La^{2l}u^1$ and integrating with respect to $x,t$, it follows
	\begin{equation}\label{4.2}
	\frac{1}{2}\|\La^{l} u^1\|_{L^2}^2 +\int^t_0 \|\frac{|\xi|^{1+l}}{(1+|\xi|^2)^{s_1/2}}\widehat{u^1}^2\|_{L^2}^2d\tau
=\frac{1}{2}\|\La^l u_0\|_{L^2}^2.	\end{equation}
Letting $l=0$ and $l=s$ in (\ref{4.2}), we have
\begin{equation}
\dis \sup_{0\leq t\leq T_0}\|u^1\|^2_{H^s}\leq  c_1 (\|u^1\|_{L^2}^2+\|u^1\|_{\dot{H}^s}^2)\leq 4c_1\|u_0\|^2_{H^s}\leq E^2.
\end{equation}
Then for any $T_0>0$, we have $u^1\in \X$.
	
We assume there exists a $T_0$ sufficiently small such that $u^j(t,x)\in \X, \forall j\leq m$. By the induction method, to complete the proof of the lemma,  it is remain to prove that  $u^{m+1}\in \X$.

	Multiplying (\ref{4.1}) with $\La^{2l}u^{m+1}$ and integrating with respect to $x$, it follows
\begin{equation}\label{4.6}
	\begin{split}
	\frac{1}{2}\frac{d}{dt}\|\La^l u^{m+1}\|_{L^2}^2+\|\frac{|\xi|^{1+l}}{(1+|\xi|^2)^{s_1/2}}\widehat{u^{m+1}}\|_{L^2}^2&\leq \int_{\R^n }|\frac{\widehat{u^{m+1}}\widehat{(u^{m})^{\theta+1}}|\xi|^{2l+1}}{(1+|\xi|^2)^{s_2}}|
	d\xi.
	\end{split}
\end{equation}

	From $s>\frac{n}{2}$ and Lemma \ref{l2.9} and Sobolev inequality, we have 
\begin{equation}\label{4.9j}
\|\La^l (u^{m})^{\theta+1}\|_{L^2}\leq c\|u^m\|_{L^\infty}^\theta \|\La^l u^m\|_{L^2}\leq c\|u^m\|_{H^s}^{\theta}\|\La^l u^m\|_{L^2}\leq c\|u^m\|_{H^s}^{\theta+1}, \forall\, {0\leq l\leq s} 
\end{equation}
which follows \begin{eqnarray}\label{4.8}
\begin{split}
 \int_{\R^n }|\frac{\widehat{u^{m+1}}\widehat{(u^{m})^{\theta+1}}|\xi|^{2l+1}}{(1+|\xi|^2)^{s_2}}|d\xi
 &	\leq \|\La^l {(u^{m})^{\theta+1}}\|_{L^2}\|\frac{|\xi|^{l+1}}{(1+|\xi|^2)^{s_2}}{u^{m+1}}\|_{L^2}\\
 &	\leq\|u^m\|_{H^s}^{\theta+1} \|\frac{|\xi|^{l+1}}{(1+|\xi|^2)^{s_2}}{u^{m+1}}\|_{L^2}\\
 &\leq c\|u^m\|_{H^s}^{2\theta+2}+\frac{1}{2}\|\frac{|\xi|^{l+1}}{(1+|\xi|^2)^{s_2}}{u^{m+1}}\|_{L^2}^2
 \\&\leq cE^{2\theta+2}+\frac{1}{2}\|\frac{|\xi|^{l+1}}{(1+|\xi|^2)^{s_1/2}}{u^{m+1}}\|_{L^2}^2,
\end{split}
\end{eqnarray}
where in the last inequality we have used the fact \[\frac{1}{2}\|\frac{|\xi|^{l+1}}{(1+|\xi|^2)^{s_2}}{u^{m+1}}\|_{L^2}^2\leq \frac{1}{2}\|\frac{|\xi|^{l+1}}{(1+|\xi|^2)^{s_1/2}}{u^{m+1}}\|_{L^2}^2 \text{ for } 2s_2\geq s_1,\]
 and \[ u^m\in \X.\]
 (\ref{4.6}) and (\ref{4.8}) immediately yield 
\begin{equation}\label{4.9}
\begin{split}
\frac{d}{dt}\|\La^l u^{m+1}\|_{L^2}^2+\|\frac{|\xi|^{l+1}}{(1+|\xi|^2)^{s_1/2}}{u^{m+1}}\|_{L^2}^2 &\leq cE^{2\theta+2}.
\end{split}
\end{equation}
Integrating the inequality with respect to $t$, for all $0\leq t\leq  T_0$ it follows
\begin{equation}\label{4.10}
\begin{split}
\|\La^l u^{m+1}\|_{L^2}^2+\int^{t}_0\|\frac{|\xi|^{l+1}}{(1+|\xi|^2)^{s_1/2}}{u^{m+1}}\|_{L^2}^2 d\tau&\leq cT_0E^{2\theta+2}+\|u_0\|_{H^s}^2.
\end{split}
\end{equation}
Taking $l=0$ and $l=s$, from (\ref{1.2}) we have  for all $0\leq t\leq T_0$
\begin{equation}
\|u^{m+1}\|_{H^s}^2\leq c_1(\|\La^s u^{m+1}\|_{L^2}^2+\| u^{m+1}\|_{L^2}^2)\leq cT_0E^{2\theta+2}+\|u_0\|_{H^s}^2. 
\end{equation}
Choosing $T_0$ sufficiently small, we deduce
$\|u^{m+1}\|_{H^s}\leq E$
that is \[u_{m+1}\in \X, \forall \,0\leq t\leq T_0.\]
\end{proof}

\begin{proposition}\label{b}
	For $T_0$ mentioned in Propositon \ref{a}, $\{u^{m}\}$  is a Cauchy sequence in $\X$.
\end{proposition}
\begin{proof}
We only need to prove that there exists a constant $0<k<1$ such that
\begin{equation}
\|u^{m+1}-u^{m}\|_{H^s}\leq k\|u^{m}-u^{m-1}\|_{H^s}.
\end{equation}

From (\ref{4.1}), we have
\begin{eqnarray}
\dis
& \dis \frac{\partial (u^{m+1}-u^{m})}{\partial
	t}-\frac{\Delta}{(I-\Delta)^{s_1}}(u^{m+1}-u^{m})=\frac{-\text{div} ((u^{m})^{\theta+1}-(u^{m-1})^{\theta+1})}{(I-\Delta)^{s_2}}, \label{4.13}\\
&u^{m+1}(0,x)-u^m(0,x)=0.\label{4.14}
\end{eqnarray}

	Multiplying (\ref{4.13}) by $\La^{2l}(u^{m+1}-u^{m})$ and integrating with respect to $x$, it follows
	\begin{equation}\label{4.15}
	\begin{split}
	&\frac{1}{2}\frac{d}{dt}\|\La^l (u^{m+1}-u^{m})\|_{L^2}^2+\|\frac{|\xi|^{1+l}}{(1+|\xi|^2)^{s_1/2}}\widehat{(u^{m+1}-u^{m})}\|_{L^2}^2\\&\ \ \ \leq \int_{\R^n }|\frac{\widehat{(u^{m+1}-u^{m})}(\widehat{(u^{m})^{\theta+1}}-\widehat{(u^{m-1})^{\theta+1}})|\xi|^{2l+1}}{(1+|\xi|^2)^{s_2}}|d\xi\\
	&\ \ \leq \|\frac{|\xi|^{1+l}}{(1+|\xi|^2)^{s_2}}\widehat{(u^{m+1}-u^{m})}\|_{L^2} \|\La^l((u^{m})^{\theta+1}-(u^{m-1})^{\theta+1})\|_{L^2}
	\\
	&\ \ \leq \frac{1}{2}\|\frac{|\xi|^{1+l}}{(1+|\xi|^2)^{s_2}}\widehat{(u^{m+1}-u^{m})}\|_{L^2}^2+\frac{1}{2} \|\La^l ((u^{m})^{\theta+1}-(u^{m-1})^{\theta+1})\|_{L^2}^2\\
	&\ \ \leq \frac{1}{2}\|\frac{|\xi|^{1+l}}{(1+|\xi|^2)^{s_1/2}}\widehat{(u^{m+1}-u^{m})}\|_{L^2}^2+\frac{1}{2} \|\La^l ((u^{m})^{\theta+1}-(u^{m-1})^{\theta+1})\|_{L^2}^2
	\end{split}
	\end{equation}
	which implies 
		\begin{equation}\label{4.16}
		\frac{d}{dt}\|\La^l (u^{m+1}-u^{m})\|_{L^2}^2+\|\frac{|\xi|^{1+l}}{(1+|\xi|^2)^{s_1/2}}\widehat{(u^{m+1}-u^{m})}\|_{L^2}^2 \leq \|\La^l ((u^{m})^{\theta+1}-(u^{m-1})^{\theta+1})\|_{L^2}^2. 
		\end{equation}
		 From $u^m\in \X$ for all $m\geq 1$
 and Lemma \ref{l2.9}, for all $s\geq l\geq 0$, we deduce
		\begin{equation}\label{4.17}
		\begin{split}
		&\| (u^{m})^{\theta+1}-(u^{m-1})^{\theta+1}\|_{\dot{H}^l}\\&
		\leq \|{(u^{m})^{\theta} (u^m-u^{m-1}  )}\|_{\dot{H}^l}+\|\Bigl((u^{m})^{\theta}-(u^{m-1})^{\theta}\Bigl)u_{m-1}\|_{\dot{H}^l}
		\\&
		\leq c\|u^{m}\|_{L^\infty}^{\theta-1}\|\La^l u^{m}\|_{L^2}\|u^{m}-u^{m-1}\|_{L^2}+\|u^{m}\|_{L^\infty}^\theta\|u^m-u^{m-1}\|_{\dot{H}^l}
		\\&\ \ +\|u^{m-1}\|_{L^\infty}\|(u^{m})^{\theta}-(u^{m-1})^{\theta}\|_{\dot{H}^l} +\|u^{m-1}\|_{\dot{H}^l}\|(u^m)^\theta-(u^{m-1})^{\theta}\|_{L^\infty}
		 \\&\leq c(\|u^m\|_{{H}^s},\|u^{m-1}\|_{{H}^s})\Bigl (\|u^m-u^{m-1}\|_{{H}^s}+\|(u^{m})^{\theta}-(u^{m-1})^{\theta}\|_{{H}^s}\Bigl )
		\end{split}
		\end{equation}
			where $c(\|u^m\|_{{H}^s},\|u^{m-1}\|_{{H}^s}) $ is a positive constant depending on $ \|u^m\|_{{H}^s},\|u^{m-1}\|_{{H}^s}$.
			Take $l=0$ and $l=s$, by using (\ref{4.18}) repeatedly，  we have
			\begin{equation}\label{4.18}
			\begin{split}
			&\|(u^{m})^{\theta+1}-(u^{m-1})^{\theta+1}\|_{\dot{H}^s}
			\\&\leq c\Bigl (\|u^m-u^{m-1}\|_{{H}^s}+\|(u^{m})^{\theta}-(u^{m-1})^{\theta}\|_{{H}^s}\Bigl ) \\&\leq c\Bigl (\|u^m-u^{m-1}\|_{{H}^s}+\|(u^{m})^{\theta-1}-(u^{m-1})^{\theta-1}\|_{{H}^s}\Bigl )\\
			\cdots
			\\&\leq c \|u^m-u^{m-1}\|_{{H}^s},
			\end{split}
			\end{equation}
			 from (\ref{4.16}) and (\ref{4.18}), it follows
				\begin{equation}\label{4.19}
				\frac{d}{dt}\bigl(\|\La^s (u^{m+1}-u^{m})\|_{L^2}^2+\| u^{m+1}-u^{m}\|_{L^2}^2\bigl) \leq c \|u^m-u^{m-1}\|_{{H}^s}^2.
				\end{equation}
		Integrating the inequality with respect to $t$, for all $0\leq t\leq  T_0$ it follows	
				\begin{equation}\label{4.20}
			\|\La^s (u^{m+1}-u^{m})\|_{L^2}^2+\| u^{m+1}-u^{m}\|_{L^2}^2 \leq c \int^{T_0}_0\|u^m-u^{m-1}\|_{{H}^s}^2d\tau,	
				\end{equation}
				which implies
				\begin{equation}\label{4.21}
				\begin{split}
		\sup_{0\leq t\leq T_0}\| (u^{m+1}-u^{m})\|_{H^s}^2 &\leq 	\sup_{0\leq t\leq T_0}	cT_0(\|\La^s (u^{m+1}-u^{m})\|_{L^2}^2+\| u^{m+1}-u^{m}\|_{L^2}^2) \\&\leq cc_1T_0\sup_{0\leq t\leq T_0}\| u^{m}-u^{m-1}\|_{H^s}^2.
			\end{split}
				\end{equation}
				Choose $T_0$ sufficiently small such that $0<k=cc_1T_0<1$, (\ref{4.15}) follows from (\ref{4.21}).
\end{proof}

From Proposition \ref{a} and Proposition \ref{b}, there exists a $u(t,x)\in \X$ which satisfies (\ref{1.1}) for $0\leq t\leq T_0$. Thus the local existence is proved. 

\section{Global existence}

In order to obtain the global existence, first we consider the bounded estimates of $\|u\|_{L^\infty}$ and $\|u\|_{H^s}$.
\begin{theorem}\label{t5.1}
	Suppose $u_0\in H^s(\R^n)$,  $s>\max\{s_2,n/2\}, n>2, s_2>s_1$,  then we can get
	\begin{equation}\label{5.1j}
	\|u\|_{L^\infty}\leq c\|u\|_{H^s}\leq C
	\end{equation}
	where $c,C$ depend on $\|u_0\|_{H^s}$.
\end{theorem}

\begin{proof}
(\ref{1.1}) is equivalent to 
\begin{equation}\label{5.1}
(I-\Delta)^{s_2}u_t-(I-\Delta)^{s_2-s_1}\Delta u=\text{div} (u^{\theta+1}).
\end{equation}

{\bf Step 1. The estimation of $\|u\|_{H^{s_2}}$.}

Multiplying (\ref{5.1}) by $u$ and integrating with respect to $x$, it follows 
\begin{equation}\label{5.2}
\int_{\R^n}(I-\Delta)^{s_2}u_t udx+\int_{\R^n}(I-\Delta)^{s_2-s_1}(-\Delta) u udx=\int_{\R^n}\text{div} (u^{\theta+1})udx.
\end{equation}
Noting the fact 
\begin{equation}\label{5.3}
\int_{\R^n}(I-\Delta)^{s_2}u_t udx=\int_{\R^n}(1+|\xi|^2)^{s_2}\hat{u}_t \hat{u}d\xi=\frac{1}{2}\frac{d}{dt}\|(1+|\xi|^2)^{s_2/2}\hat{u}\|_{L^2}^2,
\end{equation}
\begin{equation}\label{5.4}
\int_{\R^n}\text{div} (u^{\theta+1})udx=0,
\end{equation}
\begin{equation}\label{5.6}
\int_{\R^n}(I-\Delta)^{s_2-s_1}(-\Delta) u udx=\int_{\R^n}(1+|\xi|^2)^{s_2-s_1}|\xi|^2\hat{u}^2d\xi,
\end{equation}
 it follows
 \begin{equation}\label{5.7}
 \frac{1}{2}\frac{d}{dt}\|u\|_{{H}^{s_2}}^2+2\int_{\R^n}(1+|\xi|^2)^{s_2-s_1}|\xi|^2\hat{u}^2d\xi=0.
 \end{equation}
 Integrating (\ref{5.7}) with respect to $t$, it follows
 \begin{equation}\label{5.8}
 \|u\|_{{H}^{s_2}}^2+\int^t_0\int_{\R^n}(1+|\xi|^2)^{s_2-s_1}|\xi|^2\hat{u}^2d\xi d\tau=\|u_0\|_{H^{s_2}}^2,
 \end{equation}
 which implies
  \begin{equation}\label{5.9}
  \|u\|_{{H}^{s_2}}^2\leq \|u_0\|_{H^{s_2}}^2,\int^t_0\|\La^{1+s_2-s_1}u\|_{L^2}^2d\tau\leq \|u_0\|_{H^{s_2}}^2.
  \end{equation}
  
  {\bf Step 2. The estimation of $\|u\|_{L^\infty}$ and $\|u\|_{H^s}$.}
  
   
  Assume that $k_0\in \N$ satisfies that
  \begin{equation}
  \label{5.10j}
 k_0l_0+s_2<s, (k_0+1)l_0+s_2\geq s,
  \end{equation}
  where $l_0=(s_2-s_1)/2$. We claim that
  \begin{equation}\label{5.11j}
  	\|\La^{kl_0}u\|_{L^2}\leq c\|u_0\|_{H^s}, \int^t_0\|\La^{(k+2)l_0+1}u\|_{L^2}^2d\tau\leq c\|u_0\|_{H^s}^2\text{ for all } 0\leq k\leq k_0.
  \end{equation}
  
   We will prove (\ref{5.11j}) by induction.
   For $k=0$, (\ref{5.11j}) follows from Step 1.
   Assume $k=j\leq k_0-1$, (\ref{5.11j}) is satisfied,
   a.e., \begin{equation}\label{5.12}
   \|\La^{jl_0}u\|_{L^2}\leq c\|u_0\|_{H^s}, \int^t_0\|\La^{(j+2)l_0+1}u\|_{L^2}^2d\tau\leq c\|u_0\|_{H^s}^2.
   \end{equation} 
  We will consider (\ref{5.11j})  for $k=j+1$. Multiplying (\ref{5.1}) by $\La^{l}u$ and integrating with respect to $x$, similar to (\ref{5.8}),  it follows 
 \begin{equation}\label{5.10}
 \begin{split}
 \frac{1}{2}\frac{d}{dt}\|\La^l u\|_{{H}^{s_2}}^2+\int_{\R^n}(1+|\xi|^2)^{s_2-s_1}|\xi|^{2+2l}\hat{u}^2d\xi= \int_{\R^n}\text{div}(u^{\theta+1})\La^{2l}udx.
 \end{split}
 \end{equation}
 From Sobolev inequality, Lemma \ref{l2.9}, (\ref{5.12}) and $\theta\leq \theta_0$, we have
  \begin{equation}\label{5.14}
  \|\La^{(j+2)l_0}u\|_{L^{\frac{2n}{n-2}}}\leq \|\La^{(j+2)l_0+1}u\|_{L^2},
  \end{equation}
  \begin{equation}\label{5.15}
  \begin{split}
  \|\La^{jl_0+1} (u^{\theta+1})\|_{L^{\frac{2n}{n+2}}}&\leq c\|\La^{jl_0+1}u\|_{L^{p_1}} \bigl (\|u^\theta\|_{L^{p_2}}+\|u^{\theta-1}\|_{L^{p_2}}\|u\|_{L^{p_2}}
  +\cdots+\|u\|^\theta_{L^{p_2}})\\
  &\leq c \|\La^{1+(j+2)l_0}u\|_{L^{2}}\|u\|_{H^{s_2}}^\theta\leq 
  c\|\La^{1+(j+2)l_0}u\|_{L^2},
  \end{split} \end{equation}
  where $p_1=\frac{2n}{n-4l_0}, p_2=\frac{n}{1+2l_0}$ and the last inequality follows from $p_2\theta\leq \frac{2n}{n-2s_2} $ for $n>2s_2$ and $\|u^\theta\|
  _{L^{p_2}}\leq \|u\|_{H^{s_2}}^\theta $.
   Taking $l=(j+1)l_0$, integrating (\ref{5.10}) with respect to $t$, 
  it follows from (\ref{5.12}), (\ref{5.14}) and (\ref{5.15}) that
 \begin{equation}\label{5.16}
   \begin{split}
   &\frac{1}{2}\|\La^{(j+1)l_0} u\|_{{H}^{s_2}}^2+\int^t_0\int_{\R^n}(1+|\xi|^2)^{s_2-s_1}|\xi|^{2+2(j+1)l_0}\hat{u}^2d\xi d\tau\\&\ \ = \int^t_0\int_{\R^n}\text{div}(u^{\theta+1})\La^{2(j+1)l_0}udxd\tau+c\|u_0\|_{H^{s}}^2
  \\&\ \ \leq \int^t_0\|\La^{jl_0+1}(u^{\theta+1})\|_{L^{\frac{2n}{n+2}}}\|\La^{(j+2)l_0}u\|_{L^{\frac{2n}{n-2}}}d\tau+c\|u_0\|_{H^{s}}^2
  \\& \ \ \leq \int^t_0\|\La^{1+(j+2)l_0}u\|_{L^2}^2d\tau+c\|u_0\|_{H^{s}}^2\leq  c\|u_0\|_{H^{s}}^2
  \end{split}
   \end{equation}
   which implies that (\ref{5.11j}) follows for $k=j+1$.
Take $l=s-s_2$ and from (\ref{5.10j}),
similar to (\ref{5.14}) and (\ref{5.15}), we have
 \begin{equation}\label{5.17}
 \|\La^{(k_0+2)l_0}u\|_{L^{\frac{2n}{n-2}}}\leq \|\La^{(k_0+2)l_0+1}u\|_{L^2}, 
 \end{equation}
 \begin{equation}\label{5.18}
 \begin{split}
 &\|\La^{2s-2s_2-(k_0+2)l_0+1} (u^{\theta+1})\|_{L^{\frac{2n}{n+2}}}\\ &\leq c\|\La^{2s-2s_2-(k_0+2)l_0+1}u\|_{L^{p_1}} \bigl (\|u^\theta\|_{L^{p_2}}+\|u^{\theta-1}\|_{L^{p_2}}\|u\|_{L^{p_2}}
 +\cdots+\|u\|^\theta_{L^{p_2}})\\
 &\leq  c\|\La^{1+(k_0+2)l_0}u\|_{L^{2}}\|u\|_{H^{s_2}}^\theta\leq 
 c\|\La^{1+(k_0+2)l_0}u\|_{L^2},
 \end{split} \end{equation}
 where $\bar p_1= \frac{2n}{n-2\bigl (2（(k_0+2)l_0-(2s-2s_2)\bigl )},\text{} \bar p_2=\frac{n}{1+\bigl(2(k_0+2)l_0-(2s-2s_2)\bigl )}$ and the last inequality have used the fact
 \begin{equation}
 \bar{p}_1\geq p_1, \, \bar{p}_2\theta \leq p_2\theta.
 \end{equation}
For $l=s-s_2$,  integrating (\ref{5.10}) with respect to $t$, 
 it follows from (\ref{5.12}), (\ref{5.17}) and (\ref{5.18}) that
 \begin{equation}\label{5.16}
 \begin{split}
 &\frac{1}{2}\|\La^{s-s_2} u\|_{{H}^{s_2}}^2+\int^t_0\int_{\R^n}(1+|\xi|^2)^{s_2-s_1}|\xi|^{2+2s-2s_2}\hat{u}^2d\xi d\tau\\&\ \ = \int^t_0\int_{\R^n}\text{div}(u^{\theta+1})\La^{2s-2s_2}udxd\tau+c\|u_0\|_{H^{s}}^2
 \\&\ \ \leq \int^t_0\|\La^{2s-2s_2-(k_0+2)l_0+1}(u^{\theta+1})\|_{L^{\frac{2n}{n+2}}}\|\La^{(k_0+2)l_0}u\|_{L^{\frac{2n}{n-2}}}d\tau+c\|u_0\|_{H^{s}}^2
 \\& \ \ \leq \int^t_0\|\La^{1+(k_0+2)l_0}u\|_{L^2}^2d\tau\leq  c\|u_0\|_{H^{s}}^2
 \end{split}
 \end{equation}
 which implies that
\begin{equation}\label{5.22}
\|u\|_{\dot{H}^{s}}\leq c\|u_0\|_{H^{s}}.
\end{equation}
It  follows from $s>n/2$, (\ref{5.9}) and (\ref{5.22}) that
\begin{equation}\label{5.23}
\|u\|_{L^\infty}\leq c\|u\|_{{H}^{s}}\leq c\|u_0\|_{H^{s}}.
\end{equation}
\end{proof}
{\bf Proof of Theorem 1.1}
From Theorem \ref{t5.1} and the local existence,  we  derive a global solution for (\ref{1.1}) such that
\[u\in L^\infty(0,\infty;H^s(\R^n)).\]
Theorem 1.1 is complete.

\section{Decay  estimates in $\dot{H}^s(\R^n)$}
Let \begin{eqnarray}\label{j3.3j}
\chi_0(\eta)=\begin{cases}
1,&|\eta|\leq 1,\\ 
0,&|\eta|\geq 2 
\end{cases}
\end{eqnarray}
be a smooth cut-off function. Then we define the time-frequency cut-off operator $\chi(t,D)$ with the symbol $\chi(t,\xi)=\chi_0(\frac{1+t}{\mu}|\xi|^2)$, where $\mu>n+2s$. Then
\begin{eqnarray}\label{j3.3}
\chi(t,\xi)=\begin{cases}
	1,&|\xi|\leq \eta(t),\\ 
	0,&|\xi|\geq 2\eta(t)， 
\end{cases}
\end{eqnarray}
where $\eta(t)=\sqrt{\frac{\mu}{(1+t)}}$. Then for a function $g(x,t)$ we can decompose it into two parts : the low frequency part $g_L$ and the high frequency part $g_H$ where 
 \[g_L(x,t)=\chi(t,D)g(x,t),\,g_H(x,t)=(1-\chi(t,D))g(x,t).\]
For the low frequency part $u_L$ of the solution $u$ for (\ref{1.1}), we have the following decay estimates.
\begin{lemma}\label{l6.1}
	Suppose $u_0\in L^1(\R^n)\bigcap H^s(\R^n)$ and $s>\max\{s_2,n/2\}, s_2>s_1, 0\leq s_1<1, 1\leq \theta\leq \theta_0, \theta\in \N$, then we get for all $l\geq 0$
	\begin{equation}
	\|\La^{l}u_L\|_{L^2}\leq c(1+t)^{-\frac{n}{4}-\frac{l}{2}}
	\end{equation}
	where $c>0$ is a positive constant depending on $\|u_0\|_{L^1},\|u_0\|_{H^s}$.
\end{lemma}
\begin{proof}
According to (\ref{3.2}), it follows
\begin{equation}\label{6.3}
u_L=\chi(t,D)G*u_0+\int^t_0 \chi(t,D)G(t-\tau,\cdot)*\frac{\text{div} (u^{\theta+1})}{(I-\Delta)^{s_2}}d\tau.
\end{equation}
By Minkowski's inequality, we have
\begin{equation}\label{6.4}
\begin{split}
\|u_L\|_{L^2}&\leq \|\chi(t,D)G*u_0\|_{L^2}+\Bigl (\int^t_0\| \chi(t,D)G(t-\tau,\cdot)*\frac{\text{div} (u^{\theta+1})}{(I-\Delta)^{s_2}}\|_{L^2}^2d\tau\Bigl )^{1/2}\\
&=J_1+J_2.
\end{split}
\end{equation}
For $J_1$, by Theorem \ref{T3.3} and Young's inequality, we obtain
\begin{equation}
\begin{split}
J_1\leq \|\chi(t,D)G\|_{L^2}\|u_0\|_{L^1}\leq c(1+t)^{-n/4}\|u_0\|_{L^1}.
\end{split}
\end{equation}
By Lemma \ref{l2.9}  and (\ref{5.1j}), it follows
\begin{equation}
\|u^{\theta+1}\|_{L^1}\leq \|u\|_{L^2}^2\|u\|_{L^\infty}^{\theta-1}\leq c,
\end{equation}
where $c$ is a positive constant depending on $\|u_0\|_{H^s}$.
Thus for $J_2$, by Plancherel's Theorem, we know that
\begin{equation}\label{6.7}
\begin{split}
|J_2|^2&=\int^t_0\int_{\R^n}|\chi(t,\xi)|^2e^{-\frac{2|\xi|^2(t-\tau)}{(1+|\xi|^2)^{s_1}}}\frac{|\widehat{u^{\theta+1}}|^2|\xi|^2}{(1+|\xi|^2)^{s_2}}d\xi d\tau\\
&\leq c\int^t_0 \|\widehat{u^{\theta+1}}\|^2_{L^\infty}\int_{\R^n}|\xi|^2|
\chi(t,\xi)|^2d\xi d\tau
\\&\leq c\int^t_0 \|{u^{\theta+1}}\|^2_{L^1}\int_{|\xi|\leq 2\eta(t)}|\xi|^2d\xi d\tau\\
&\leq c\int^t_0 (1+\tau)^{-\frac{n+2}{2}} d\tau\leq c(1+t)^{-\frac{n}{2}}
\end{split}
\end{equation}
where $c$ depends on $\|u_0\|_{H^s}, \|u_0\|_{L^1}$.
Then we have
\begin{equation}
	\| u_L\|^2\leq c (1+t)^{-\frac{n}{2}}.
\end{equation}
Thus for all $l\geq 0$ \begin{equation}\label{6.8}
\begin{split}
\|\La^l u_L\|^2&=\int_{\{|\xi|\leq 2\eta(t)\}}|\chi(t,\xi)|^2|\xi|^{2l}|\hat{u}_L|^2d\xi 
\leq c(1+t)^{-l}\|u_L\|_{L^2}^2
\leq c (1+t)^{-\frac{n}{2}-l}.
\end{split}
\end{equation}
\end{proof}
\begin{theorem}\label{t6.1} Under the assumptions of Lemma \ref{l6.1},  it follows
	\begin{equation}\label{6.11j}
	\|u\|_{H^{s_2}}\leq c(1+t)^{-\frac{n}{4}},
	\end{equation}
	\begin{equation}\label{6.12j}
	\int^t_{t_0}(1+\tau)^{\mu/2}\int_{\R^n}(1+|\xi|^2)^{s_2-s_1}|\xi|^2\hat{u}^2d\xi d\tau \leq c(1+t)^{-\frac{n-\mu}{2}}
	\end{equation}
where $c$ depends on $\|u_0\|_{H^s}, \|u_0\|_{L^1}$.	
\end{theorem}
\begin{proof}
Taking  
	$\frac{\mu}{1+t_0}=\frac{1}{4}$, then for all $t\geq t_0>0$, it gives
	\begin{equation}
	\begin{split}
     \eta(t)\leq 1/2,
	\end{split}
	\end{equation}
	which follows
	\begin{equation}\label{6.13}
	\begin{split}
&\int_{\R^n}(1+|\xi|^2)^{s_2-s_1}|\xi|^2\hat{u}^2d\xi\\&\ \ \geq \int_{|\xi|\geq \eta(t)}(1+|\xi|^2)^{s_2}\frac{|\xi|^2}{(1+|\xi|^2)^{s_1}}\hat{u}^2d\xi
\\&\ \ \geq \int_{|\xi|\geq 1}(1+|\xi|^2)^{s_2}\frac{|\xi|^2}{(1+|\xi|^2)^{s_1}}\hat{u}^2d\xi+\int_{1\geq|\xi|\geq \eta(t)}(1+|\xi|^2)^{s_2}\frac{|\xi|^2}{(1+|\xi|^2)^{s_1}}\hat{u}^2d\xi
\\&\ \ \geq 
\int_{|\xi|\geq 1}(1+|\xi|^2)^{s_2}\frac{|\xi|^{2-2s_1}}{2^{s_1}}\hat{u}^2d\xi+\int_{1\geq|\xi|\geq \eta(t)}(1+|\xi|^2)^{s_2}\frac{|\xi|^2}{2^{s_1}}\hat{u}^2d\xi
\\&\ \  \geq  \frac{1}{2} \int_{|\xi|\geq 1}(1+|\xi|^2)^{s_2}\hat{u}^2d\xi+\frac{1}{2}\eta(t)^2\int_{1\geq|\xi|\geq \eta(t)}(1+|\xi|^2)^{s_2}\hat{u}^2d\xi
\\& \ \ \geq \frac{\eta(t)^2}{2}\int_{|\xi|\geq \eta(t)}(1+|\xi|^2)^{s_2}\hat{u}^2d\xi
\\& \ \ =\frac{\eta(t)^2}{2}\Bigl (\|u\|_{H^{s_2}}^2-\int_{|\xi|\leq  \eta(t)}(1+|\xi|^2)^{s_2}\hat{u}^2d\xi\Bigl).
	\end{split}
	\end{equation}
	From (\ref{5.7}), (\ref{6.13})  and Lemma \ref{l6.1}  we have
	\begin{equation}\label{6.14}
	\begin{split}
\frac{d}{dt}\|u\|_{H^{s_2}}^2+\int_{\R^n}(1+|\xi|^2)^{s_2-s_1}|\xi|^2\hat{u}^2d\xi+\frac{1}{2}\eta(t)^2\|u\|_{H^{s_2}}^2
&\leq \frac{1}{2}\eta(t)^2\int_{|\xi|\leq  \eta(t)}(1+|\xi|^2)^{s_2}\hat{u}^2d\xi
\\&\leq c\eta(t)^2\|u_L\|_{H^{s_2}}^2\leq c(1+t)^{-\frac{n}{2}-1}
	\end{split}
	\end{equation}
which follows 
\begin{equation}\label{6.15}
\begin{split}
\frac{d}{dt}\|u\|_{H^{s_2}}^2+\int_{\R^n}(1+|\xi|^2)^{s_2-s_1}|\xi|^2\hat{u}^2d\xi+\frac{1}{2}\mu(1+t)^{-1}\|u\|_{H^{s_2}}^2
\leq  c(1+t)^{-\frac{n}{2}-1}.
\end{split}
\end{equation}
Multiplying (\ref{6.15}) by $e^{\int^t_0 \mu/2 (1+\tau)^{-1}d\tau}=(1+t)^{\mu/2}$ and 
integrating  from $t_0$ to $t$, from Theorem \ref{t5.1} we have
\begin{equation}\label{6.16}
\begin{split}
	&(1+t)^{\mu/2}\|u\|_{H^{s_2}}^2 +\int^t_{t_0}(1+\tau)^{\mu/2}\int_{\R^n}(1+|\xi|^2)^{s_2-s_1}|\xi|^2\hat{u}^2d\xi d\tau
	\\&\ \ \leq  \int^t_{t_0}(1+\tau)^{\mu/2-\frac{n}{2}-1}d\tau +c\|u(t_0)\|_{H^{s_2}}^2
	\leq c(1+t)^{-\frac{n}{2}+\mu/2}.
\end{split}
\end{equation}

Then (\ref{6.11j}) and (\ref{6.12j}) follow from (\ref{6.16}).
\end{proof}
\begin{theorem}\label{l6.2}Suppose $u_0\in L^1(\R^n)\bigcap H^s(\R^n)$ and $s>\max\{s_2,n/2\}, s_2>s_1, 0\leq s_1< 1$
	\begin{equation}
	\|u\|_{L^1}\leq c\|u_0\|_{L^1}.
	\end{equation}
\end{theorem}
\begin{proof}
	According to (\ref{3.2}),  we have
	\begin{equation}\label{6.19j}
	\begin{split}
	\|u\|_{L^1}&\leq \|G*u_0\|_{L^1}+\int^t_0\| G(t-\tau,\cdot)*\frac{\text{div} (u^{\theta+1})}{(I-\Delta)^{s_2}}\|_{L^1}d\tau\\
	&=I_1+I_2.
	\end{split}
	\end{equation}
	For $I_1$, by Young's inequality and Theorem  \ref{T3.3}, it follows
	\begin{equation}\label{6.20jj}
	I_1\leq \|G\|_{L^1}\|u_0\|_{L^1}\leq c.
	\end{equation}
	For $I_2$, notice the fact that
	\begin{equation}
	(I-\Delta)^{-s_2}u^{\theta+1}=k(x)*u^{\theta+1}
	\end{equation}
	where 
	$k(x)=(4\pi)^{-\frac{n}{2}}\Gamma(s_2)^{-1}\int^\infty_0 t^{\frac{2s_2-n-2}{2}}e^{-t-\frac{|x|^2}{4t}}d\tau$ (refer to （(1.8)） in \cite{AU}) satisfying
	\begin{equation}
	\|k(x)*u^{\theta+1}\|_{L^p}\leq \|u^{\theta+1}\|_{L^p}, \text{ for all }p\geq 1,
	\end{equation}
	and from Theorem \ref{t6.1}, it gives
	\begin{equation}
	\|u^{\theta+1}\|_{L^1}\leq \|u\|_{L^\infty}^{\theta-1} \|u\|^2_{L^2}\leq c(1+t)^{-n/2}.
	\end{equation}
	Then by Young's inequality and Theorem  \ref{T3.3}
\begin{equation}
\begin{split}
I_2&\leq \int^t_0\| \nabla G(t-\tau,\cdot)* \frac{ u^{\theta+1}}{(I-\Delta)^{s_2}}\|_{L^1}d\tau\\
&\leq c\int^t_0\|\nabla G(t-\tau,\cdot)\|_{L^1}\|u^{\theta+1}\|_{L^1} d\tau\\
&\leq c\int^t_0(t-\tau)^{-1/2}(1+\tau)^{-n/2}d\tau
\\&=I_3.
\end{split}
\end{equation}
Since for $t\geq 1$
\begin{equation}
I_3\leq ct^{-1/2}\int^{t/2}_0(1+\tau)^{-n/2}d\tau+c(1+t)^{-n/2}\int^t_{t/2}(t-\tau)^{-1/2}d\tau\leq c
\end{equation}
and for $t\leq 1$
\begin{equation}
I_3\leq \int^t_0(t-\tau)^{-1/2}d\tau\leq c,
\end{equation}
then
\begin{equation}\label{6.26j}
I_2\leq c.
\end{equation}	
From (\ref{6.19j})	(\ref{6.20jj})	and (\ref{6.26j}), we complete the proof of Lemma \ref{l6.2}.
\end{proof}
{\bf Proof of Theorem 1.2}
	Applying $1-\chi(t,D)$ to both side of (\ref{1.1j}), we have
	\begin{equation}\label{6.20j}
(1-\chi(t,D))	u_t-\frac{\Delta}{(I-\Delta)^{s_1}}u_H=(1-\chi(t,D))\frac{-\text{div} u^{\theta+1}}{(I-\Delta)^{s_2}}.
	\end{equation}
	Noting that
	\begin{equation}\label{6.20}
	\begin{split}
&	\int_{\R^n}\La^s u_t (1-\chi(t,D))\La^s u_Hdx\\&=\int_{\R^n}\frac{d}{dt}\Bigl(\La^s u (1-\chi(t,D))\Bigl)\La^s u_Hdx-\int_{\R^n}\frac{d}{dt}\Bigl( 1-\chi(t,D)\Bigl)\La^s u\La^s u_Hdx\\
&=\int_{\R^n}\frac{d}{dt}\Bigl(\La^s u_H\Bigl)\La^s u_Hdx-\int_{\R^n}\frac{d}{dt}\Bigl( 1-\chi(t,D)\Bigl)\La^s u\La^s u_Hdx\\
&=\frac{1}{2}\frac{d}{dt}\|\La ^s u_H\|_{L^2}^2-R(t)
	\end{split}
	\end{equation}
	where 
\[R(t)=\int_{\R^n}\frac{d}{dt}\Bigl( 1-\chi(t,D)\Bigl)\La^s u\La^s u_Hdx. \]

	Multiplying (\ref{6.20j}) with $\La^{2s}u_H$, it follows from (\ref{6.20}) that for all $t\geq t_0=4\mu-1$
	\begin{equation}\label{6.22}
	\begin{split}
	&\frac{1}{2}\frac{d}{dt}\|\La^s u_H\|_{L^2}^2+\|\frac{|\xi|^{s+1}}{(1+|\xi|^2)^{s_1/2}}\widehat{u_H}\|_{L^2}^2\\&\ \ \leq \int_{{|\xi|\geq \eta(t)} }|\frac{\widehat{u_H}\widehat{u^{\theta+1}}(1-\chi(t,\xi))|\xi|^{2s+1}}{(1+|\xi|^2)^{s_2}}|d\xi+R(t)
	\\ & \ \ =  \int_{|\xi|\geq 1 }|\frac{\widehat{u_H}\widehat{u^{\theta+1}}(1-\chi(t,\xi))|\xi|^{2s+1}}{(1+|\xi|^2)^{s_2}}|d\xi+ \int_{1\geq |\xi|\geq \eta(t) }|\frac{\widehat{u_H}\widehat{u^{\theta+1}}(1-\chi(t,\xi))|\xi|^{2s+1}}{(1+|\xi|^2)^{s_2}}|d\xi+R(t)
	\\&\ \ =T_1(t)+T_2(t)+R(t).
	\end{split}
	\end{equation}
	For $T_1(t)$, from Lemma \ref{l2.5} and $s>s_2>s_1$,  it follows
	\begin{equation}\label{6.22j}
	\begin{split}
T_1(t)&\leq \int_{|\xi|\geq 1 }|\widehat{u_H}\widehat{u^{\theta+1}}||\xi|^{2s+1-2s_2}d\xi\\
&\ \ \leq \|\La^{s+1-s_1} 
u^{\theta+1}\|_{L^2}\||\xi|^{s-2s_2+s_1}\widehat{u_H}\|_{L^2(|\xi|\geq 1)}
\\&	\ \ \leq \|u\|^\theta_{L^\infty}\|\La^{s+1-s_1} u\|_{L^2}\||\xi|^{s-s_1+1}\widehat{u_H}\|_{L^2(|\xi|\geq 1)}^a \|\widehat{u_H}\|_{L^2(|\xi|\geq 1)}^{1-a}\\&
\ \  \leq c\bigl(\||\xi|^{s+1-s_1} \hat u\|_{L^2(|\xi|\geq 1)}+\||\xi| \hat u\|_{L^2(|\xi|\leq 1)}\bigl)\bigl(\va\||\xi|^{s-s_1+1}\widehat{u_H}\|_{L^2(|\xi|\geq 1)}+ c\|\widehat{u_H}\|_{L^2(|\xi|\geq 1)}\bigl)
\\& \ \ \leq \frac{1}{8} \||\xi|^{s+1-s_1} \hat u\|_{L^2(|\xi|\geq 1)}^2
+c\||\xi| \hat u\|_{L^2(|\xi|\leq 1)}^2+c\|\widehat{u_H}\|_{L^2(|\xi|\geq 1)}^2
	\end{split}
	\end{equation}	
		where $0<a<1$ and the last inequality holds by choosing $\va>0$ small enough.
		For $T_2(t)$, from Lemma \ref{l2.5} and $s>\frac{n}{2}>1>s_1$, it follows
	\begin{equation}\label{6.23}
	\begin{split}
T_2(t)&\ \ \leq \int_{1\geq |\xi|\geq \eta(t) }|\widehat{u_H}\widehat{u^{\theta+1}}||\xi|^{2s+1}d\xi\\
	&\ \ \leq \|\La^{s} 
	u^{\theta+1}\|_{L^2}\||\xi|^{s+1}\widehat{u_H}\|_{L^2(1\geq |\xi|\geq \eta(t))}
	\\&	\ \ \leq \|u\|^\theta_{L^\infty}\|\La^{s} u\|_{L^2}\||\xi|^{s+1}\widehat{u_H}\|_{L^2(1\geq |\xi|\geq \eta(t))}
	\\&
	\ \  \leq c	(\||\xi|^{s}\hat u\|_{L^2(|\xi|\geq 1)}+\||\xi|^{s}\hat u\|_{L^2(|\xi|\leq 1)})
	\||\xi|^{s+1}\widehat{u_H}\|_{L^2(1\geq |\xi|\geq \eta(t))}
	 \\&	\ \  \leq c	(\||\xi|^{s+1-s_1}\hat u\|^{a_1}_{L^2(|\xi|\geq 1)}\|\hat u\|_{L^2(|\xi|\geq 1)}^{1-a_1}+\||\xi|^{s} \hat u\|_{L^2(|\xi|\leq 1)})
	 \||\xi|^{s+1}\widehat{u_H}\|_{L^2(1\geq |\xi|\geq \eta(t))}
	 \\ & 
	 \ \ 
	\leq c
	\bigl(\va\||\xi|^{s-s_1+1}\widehat{u_H}\|_{L^2(|\xi|\geq 1)}+ c\|\widehat{u_H}\|_{L^2( |\xi|\geq 1)}+\||\xi| \hat u|_{L^2(|\xi|\leq 1)}\bigl)\||\xi|^{s+1}\widehat{u_H}\|_{L^2(1\geq |\xi|\geq \eta(t))}
	\\& \ \ \leq \frac{1}{8} \||\xi|^{s+1-s_1} \widehat{u_H}\|_{L^2( |\xi|\geq 1)}^2+\frac{1}{8} \||\xi|^{s+1} \widehat{u_H}\|_{L^2(1\geq |\xi|\geq \eta(t))}^2
	\\&\ \ \ \ +c\||\xi| \hat u\|_{L^2(|\xi|\leq 1)}^2+c\|\widehat{u_H}\|_{L^2(|\xi|\geq 1)}^2
	\end{split}
	\end{equation}	
	where $0<a_1<1$ and the last inequality holds by choosing $\va>0$ small enough.
	Then from (\ref{6.22j}) and (\ref{6.23}), we have
	\begin{equation}\label{6.25}
	\begin{split}
	&\int_{ |\xi|\geq \eta(t) }|\frac{\widehat{u_H}\widehat{u^{\theta+1}}(1-\chi(t,D))|\xi|^{2s+1}}{(1+|\xi|^2)^{s_2}}|d\xi
	\\& \ \ \leq \frac{1}{4} \||\xi|^{s+1-s_1} \widehat{u_H}\|_{L^2( |\xi|\geq 1)}^2+\frac{1}{4} \||\xi|^{s+1} \widehat{u_H}\|_{L^2(1\geq |\xi|\geq \eta(t))}^2
	\\&\ \ \ \ +c\||\xi| \hat{u}\|_{L^2(|\xi|\leq 1)}^2+c\|\widehat{u_H}\|_{L^2(|\xi|\geq 1)}^2
	\end{split}
	\end{equation}	
	and 
		\begin{equation}\label{6.26}
		\begin{split}
		&\int_{|\xi|\geq \eta(t)}\frac{|\xi|^{2+2s}}{(1+|\xi|^2)^{s_1}}\widehat{u_H}^2d\xi
		\\&\ \ = \int_{|\xi|\geq 1}\frac{|\xi|^{2+2s}}{(1+|\xi|^2)^{s_1}}\widehat{u_H}^2d\xi+\int_{1\geq|\xi|\geq \eta(t)}\frac{|\xi|^{2+2s}}{(1+|\xi|^2)^{s_1}}\widehat{u_H}^2d\xi
		\\&\ \ \geq \frac{1}{2} \||\xi|^{s+1-s_1} \widehat{u_H}\|_{L^2( |\xi|\geq 1)}+\frac{1}{2} \||\xi|^{s+1} \widehat{u_H}\|_{L^2(1\geq |\xi|\geq \eta(t))}.
		\end{split}
		\end{equation}
		Together with (\ref{6.22}), (\ref{6.25})and (\ref{6.26}), we have
	\begin{equation}\label{6.27}
	\begin{split}
	&\frac{1}{2}\frac{d}{dt}\|\La^s u_H\|_{L^2}^2+\frac{1}{4} \||\xi|^{s+1-s_1} \widehat{u_H}\|_{L^2( |\xi|\geq 1)}^2+\frac{1}{4} \||\xi|^{s+1} \widehat{u_H}\|_{L^2(1\geq |\xi|\geq \eta(t))}^2
	\\ & \ \ \leq c\||\xi| \hat{u}\|_{L^2(|\xi|\leq 1)}^2+c\|\widehat{u_H}\|_{L^2(|\xi|\geq 1)}^2+R(t).
	\end{split}
	\end{equation}
	 From Theorem \ref{l6.2}
	\begin{equation}\label{6.37}
	\begin{split}
	R(t)\leq &\int_{\R^n}|\hat{u}|^2|\xi|^{2s}\Bigl(1-\chi(\mu^{-1}(1+t)|\xi|^2)\Bigl)_t(1-\chi(\mu^{-1}(1+t|\xi|^2)))d\xi\\
&	\leq c\|\hat{u}\|_{L^\infty}^2\int_{\R^n}|\xi|^{2s+2}(1-\chi(z))|\frac{d}{dz}\chi(z) |\mu^{-1}d\xi\\
&	\leq c\|u\|_{L^1}^2\int_{\eta(t)\leq |\xi|\leq 2\eta(t)}|\xi|^{2s+2}d\xi\\
	&\leq  c(1+t)^{-s-n/2-1}.
	\end{split}
	\end{equation}
	Obviously,
		\begin{equation}\label{6.38}
		\begin{split}
		&\frac{1}{4} \||\xi|^{s+1-s_1} \widehat{u_H}\|^2_{L^2( |\xi|\geq 1)}+\frac{1}{4} \||\xi|^{s+1} \widehat{u_H}\|^2_{L^2(1\geq |\xi|\geq \eta(t))}\\&\ \ \geq \frac{1}{4} \||\xi|^{s} \widehat{u_H}\|_{L^2( |\xi|\geq 1)}^2+\frac{1}{4}\eta(t)^2 \||\xi|^{s} \widehat{u_H}\|_{L^2(1\geq |\xi|\geq \eta(t))}^2\\
		&\ \ \geq \frac{1}{4}\eta(t)^2 \||\xi|^{s} \widehat{u_H}\|_{L^2}^2.
		\end{split}
		\end{equation}
		From (\ref{6.27}), (\ref{6.37}) and (\ref{6.38}), it gives
		\begin{equation}\label{6.39}
		\begin{split}
		&\frac{d}{dt}\|\La^s u_H\|_{L^2}^2+\frac{\eta(t)^2}{2}  \||\xi|^{s} \widehat{u_H}\|_{L^2}^2 \leq c\||\xi| \hat{u}\|_{L^2(|\xi|\leq 1)}^2+c\|\widehat{u_H}\|_{L^2(|\xi|\geq 1)}^2+c(1+t)^{-s-n/2-1}.
		\end{split}
		\end{equation}
		Multiplying (\ref{6.39}) by $(1+t)^{\mu/2}$, then
		integrating  from $t_0=4\mu-1$ to $t$, since
		\begin{equation}
	\|u(t_0)\|_{H^s}^2\leq c\|u_0\|_{H^s}^2,
		\end{equation}
		we have
		\begin{equation}\label{6.40jj}
		\begin{split}
		(1+t)^{\mu/2}\|\La^s u_H\|_{L^2}^2&\leq c\int^t_{t_0}(1+\tau)^{\mu/2}(\||\xi| \hat{u}\|_{L^2(|\xi|\leq 1)}^2+\|\widehat{u_H}\|_{L^2(|\xi|\geq 1)}^2)d\tau\\&
		\ \ +c\int^t_{t_0}(1+\tau)^{\mu/2-s-n/2-1}d\tau+ c\|u_0\|^2_{H^s}.
		\end{split}
		\end{equation}	
		From (\ref{6.12j}) and (\ref{6.40jj}) and
		\begin{equation}\label{6.41}
		\||\xi| \hat{u}\|_{L^2(|\xi|\leq 1)}^2+\|\widehat{u_H}\|_{L^2(|\xi|\geq 1)}^2\leq \int_{|\xi|\leq 1}(1+|\xi|^2)^{s_2-s_1}|\xi|^2\hat{u}^2d\xi + \int_{|\xi|\geq 1}(1+|\xi|^2)^{s_2-s_1}|\xi|^2\hat{u}^2d\xi,
		\end{equation}
		it follows
			\begin{equation}\label{6.40j}
			\begin{split}
			(1+t)^{\mu/2}\|\La^s u_H\|_{L^2}^2&\leq c +c\int^t_{t_0}(1+\tau)^{\mu/2-s-n/2-1}d\tau+ c\|u_0\|^2_{H^s}.
			\end{split}
			\end{equation}
			which is equivalent to 
				\begin{equation}\label{6.40}
				\begin{split}
				\|\La^s u_H\|_{L^2}^2&\leq c (1+t)^{-s-n/2} +c(1+t)^{-\mu/2}.
				\end{split}
				\end{equation}
				where $c$  depends on $\|u_0\|_{H^s}$ and $\|u_0\|_{L^1}$.
				Theorem 1.2 	is proved because $\mu>n+2s$.



\begin{thebibliography}{99}
		\bibitem{AU}Aliev I, Uyhan-Bayrakci S. On inversion of Bessel potentials associated with the Laplace-Bessel differential operator. \textit{ Acta Mathematica Hungarica}, 2002; \textbf{95}(1-2): 125-145.   DOI: 10.1023/A:1015620402251.                          
			\bibitem{BBM}Benjamin T B, Bona J L, Mahony J J. Model equations for long waves in nonlinear dispersive systems. \textit{Philosophical Transactions of the Royal Society of London A: Mathematical, Physical and Engineering Sciences}, 1972; \textbf{272}(1220): 47-78. DOI: 10.1098/rsta.1972.0032.
			\bibitem{DR}Duan R, Ruan L, Zhu C. Optimal decay rates to conservation laws with diffusion-type terms of regularity-gain and regularity-loss. \textit{Mathematical Models and Methods in Applied Sciences}, 2012; \textbf{22}(07): 1250012. DOI: 10.1142/S0218202512500121.
\bibitem{JLF} Jin L, Li L, Fang S. The global existence and time-decay for the solutions of the fractional pseudo-parabolic equation. \textit{Computers and Mathematics with Applications}, 2017; \textbf{73}(10): 2221-2232 DOI: 10.1016/j.camwa.2017.03.005.



	\bibitem{K}Karch G. Lp-decay of solutions to dissipative-dispersive perturbations of conservation laws. \textit{Annales Polonici Mathematici}, 1997; \textbf{67}(1): 65-86.







 














\bibitem{J}Ju N. Existence and Uniqueness of the Solution to the Dissipative 2D Quasi-Geostrophic Equations in the Sobolev Space. \textit{Communications in Mathematical Physics}, 2004; \textbf{251}(2): 365-376. DOI: 10.1007/s00220-004-1062-2.





\bibitem{R}Rosenau P.  Extending hydrodynamics via the regularization of the Chapman-Enskog expansion. \textit{Physical Review A}, 1989; \textbf{40}(12): 7193.
DOI：10.1103/PhysRevA.40.7193.
	\bibitem{S} Stanislavova M. On the global attractor for the damped Benjamin-Bona-Mahony equation. \textit{Discrete and Continuous Dynamical Systems}, Series A, 2005; \textbf{2005}: 824-832.  DOI:10.3934/proc.2005.2005.824.

 \bibitem{S2}Strichartz R. Multipliers on fractional Sobolev spaces. Journal of Mathematics  Mechanic, 1967 (16), 1031-1060. 
	\bibitem{SSW}Stanislavova M, Stefanov A, Wang B. Asymptotic smoothing and attractors for the generalized Benjamin–Bona–Mahony equation on $R^3$. \textit{Journal of Differential Equations}, 2005; \textbf{219}(2): 451-483. DOI: 10.1016/j.jde.2005.08.004.
 \bibitem{T1} Triebel H. \textit{Interpolation theory, function spaces, differential operators}. North-Holland Pub. Co, 1978.
\bibitem{XX}Xu H, Xin G. Global Existence and Decay Rates of Solutions of Generalized Benjamin-Bona-Mahony Equations in Multiple Dimensions. \textit{Acta Mathematica Vietnamica}, 2014; \textbf{39}(2): 121-131. DOI: 10.1007/s40306-014-0054-3.

	\bibitem{WY}Wang W, Yang T. The pointwise estimates of solutions for Euler equations with damping in multi-dimensions. \textit{Journal of Differential Equations}, 2001; \textbf{173}(2): 410-450. DOI: 10.1006/jdeq.2000.3937.





















\end{thebibliography}
\end{document}